\documentclass[10pt,reqno]{amsart}
\usepackage{amsmath,amssymb}
\usepackage{longtable,lscape}
\usepackage{tabularx}
\usepackage{array}
\usepackage{layout}
\usepackage[colorlinks=true, urlcolor=black, linkcolor=black, citecolor=black]{hyperref}
\setbox0=\hbox{$+$}
\newdimen\plusheight
\plusheight=\ht0
\def\+{\;\lower\plusheight\hbox{$+$}\;}

\setbox0=\hbox{$-$}
\newdimen\minusheight
\minusheight=\ht0
\def\-{\;\lower\minusheight\hbox{$-$}\;}

\setbox0=\hbox{$\cdots$}
\newdimen\cdotsheight
\cdotsheight=\plusheight
\def\cds{\lower\cdotsheight\hbox{$\cdots$}}

\renewcommand{\(}{\left\(}
\renewcommand{\)}{\right\)}

\renewcommand{\pmod}[1]{\,(\textup{mod}\,#1)}
\theoremstyle{plain}
\newtheorem{thm}{Theorem}[section]
\newtheorem{lem}[thm]{Lemma}

\newtheorem{cor}[thm]{Corollary}

\newtheorem{conj}[thm]{Conjecture}
\newtheorem{rem}[thm]{Remark}

\newenvironment{pf}
{\vskip 0.15in \par\noindent{\bf Proof of Theorem}\hskip 0.5em\ignorespaces}
{\hfill $\Box$\par\medskip}

\begin{document}
\title{On some conjectural determinants of Sun involving residues}

\author{Rituparna Chaliha}
\address{Department of Science and Mathematics, Indian Institute of Information Technology Guwahati, Kamrup-781015, Assam, INDIA}
\email{rituparnachaliha5@gmail.com}
\author{Gautam Kalita}
\address{Department of Mathematics, Bhattadev University, Bajali, Pathsala, Assam-781325, INDIA}
\email{haigautam@gmail.com}
\date{}

\begin{abstract} 
 For an odd prime $p$ and integers $d, k, m$ with gcd$(p,d)=1$ and $2\leq k\leq \frac{p-1}{2}$, we consider the determinant  
    \begin{equation*}
        S_{m,k}(d,p) = \left|(\alpha_i - \alpha_j)^m\right|_{1 \leq i,j \leq \frac{p-1}{k}},
    \end{equation*}
    where  $\alpha_i$ are distinct $k$-th power residues modulo $p$. In this paper, we deduce some residue properties for the determinant $S_{m,k}(d,p)$ as a generalization of certain results of Sun. Using these, we further prove some conjectures of Sun related to $$\left(\frac{\sqrt{S_{1+\frac{p-1}{2},2}(-1,p)}}{p}\right) \text{ and } \left(\frac{\sqrt{S_{3+\frac{p-1}{2},2}(-1,p)}}{p}\right).$$ In addition, we investigate the number of primes $p$ such that $p\ |\ S_{m+\frac{p-1}{k},k}(-1,p)$, and confirm another conjecture of Sun related to $S_{m+\frac{p-1}{2},2}(-1,p)$.    
\end{abstract}

\maketitle

\section{Introduction and statements of results}
Let $n$ be a positive integer and $R$ a commutative ring with unity. For an $n\times n$ matrix $M=[a_{ij}]_{1\leq i,j\leq n}$ with $a_{ij}\in R$, we denote the determinant by $|M|$ or $|[a_{ij}]_{1\leq i,j\leq n}|$. Let $p$ be an odd prime and $\chi_\ell$ denotes a multiplicative character of order $\ell$ modulo $p$. For example, $\chi_2(\cdot)=\left(\frac{\cdot}{p}\right)$ is the usual Legendre symbol. In this paper, we study some conjectural determinants involving residues. Determinants with Legendre symbol entries were first considered by Lehmer  \cite{Lehmer1}, where he used a general method to determine the characteristic roots of two classes of matrices to evaluate their determinants. Extending a result of Lehmer \cite{Lehmer1}, Carlitz \cite{Carlitz1} obtained the characteristic polynomial of the matrix $$M = \left|\left(\frac{i - j}{p}\right)\right|_{1 \leq i,j \leq p-1},$$ and deduced that $|M| = p^{\frac{p-3}{2}}$. In \cite{Chapman1}, Chapman studied some matrices that appeared in the investigation of lattices constructed from quadratic residue codes and their generalizations, and determined values of 
$$\left|\left(\frac{i+j-1}{p}\right)\right|_{1\leq i,j\leq \frac{p-1}{2}} \text{ and } \left|\left(\frac{i+j-1}{p}\right)\right|_{1\leq i,j\leq \frac{p+1}{2}}$$ using quadratic Gauss sums. Following these, Vsemirnov \cite{Vsemirnov1,Vsemirnov2} used a sophisticated matrix decomposition to confirm a challenging conjecture of Chapman \cite{Chapman2} on the determinant $$\left|\left(\frac{j-i}{p}\right)\right|_{1\leq i,j\leq \frac{p+1}{2}}.$$ 
In \cite{Sun1}, Sun concentrated on determinants of the form 
\begin{equation*}
    \left|\left(\frac{f(i,j)}{p}\right)\right|_{1 \leq i,j \leq \frac{p-1}{2}},
\end{equation*}
where $f(x,y)$ is a quadratic form, and investigated their quadratic residue properties. In particular, for $p\nmid d$, Sun \cite{Sun1} studied the determinant 
\begin{equation*}
    S(d,p) = \left|\left(\frac{i^2+dj^2}{p}\right)\right|_{1 \leq i,j \leq \frac{p-1}{2}},
\end{equation*}
and proved that 
\begin{equation*}
    \left(\frac{S(d,p)}{p}\right) =
        \begin{cases}
            \left(\frac{-1}{p}\right), &\text{if $\left(\frac{d}{p}\right)=1$;}\\
             0,   &\text{if $\left(\frac{d}{p}\right)=-1$.} \\
        \end{cases}
    \end{equation*}
In addition, Sun \cite{Sun1} also posed a number of conjectures related to the determinant $S(d,p)$. In recent years, some of these conjectures and their generalizations have been proved by many mathematicians, such as Krachun et. al. \cite{Krachun1}, Grinberg et. al. \cite{Grinberg1}, Wu \cite{Wu2}, and Wu-Wang \cite{Wu1}. 
\par For $n,k\in\mathbb{N}$ with $2 \leq k \leq \frac{p-1}{2}$, throughout the paper, let $p$ be an odd prime such that $p\equiv 1$ $($mod $k)$ and $$S_{n,k}(d,p) = \left|(\alpha_i +d \alpha_j)^n\right|_{1 \leq i,j  \leq \frac{p-1}{k}},$$ where $\alpha_i$ are distinct $k$-th residues modulo $p$. Note that $S_{\frac{p-1}{2},2}(d,p)\equiv S(d,p) \pmod{p}$. From \cite[Theorem 1.8]{Kalita1}, it is known that for $p>3$,
\begin{equation*}
		\left(\frac{S_{\frac{p-3}{2},2}\left(d,p\right)}{p}\right)=\begin{cases}
		\left(\frac{d}{p}\right)^{\frac{p-1}{4}}, & \text{if}\ p \equiv 1 \pmod{4};\\ \notag
		\left(\frac{d}{p}\right)^{\frac{p-3}{4}}\left(-1\right)^{|\{0<k<\frac{p}{2}:\left(\frac{k}{p}\right)=-1\}|}, & \text{if}\ p \equiv 3 \pmod{4}.
	\end{cases}
	\end{equation*}
Moreover, for $n < \frac{p-3}{2}$
, we have from \cite[Lemma $9$]{Krattenthaler1} that $$S_{n,2}(d,p) = 0.$$
In \cite{Sun2}, Sun considered the determinant $S_{n,2}(d,p)$ for $n>\frac{p-1}{2}$, and proved a number of results related to them under the condition $\left(\frac{d}{p}\right)=-1$. Following these, Wu, She and Wang \cite{Wu3} proved a conjecture of Sun \cite[Conjecture 4.5]{Sun1} related the Legendre symbol $\left(\frac{S_{\frac{p+1}{2},2}\left(d,p\right)}{p}\right)$. Recently, Ren and Sun \cite{Ren1} studied the determinant $S_{n,2}(d,p)$ for $n>\frac{p-1}{2}$ under the condition $\left(\frac{d}{p}\right)=1$. In the following theorems, we generalize some of the results of Ren and Sun \cite{Ren1} to the determinant $S_{n,k}(d,p)$.
\begin{thm}\label{thm1} Let $p$ be an odd prime and  $k\in \mathbb{N}$ with $2\leq k\leq \frac{p-1}{2}$ such that $p>2k+1$ and $p\equiv1$ $($mod $k)$. If $n \in \{\frac{p-1}{k}+1, \frac{p-1}{k}+2, \cdots, \frac{2(p-1)}{k}-1\}$ such that $n \equiv \frac{p-1}{k} \pmod{2}$ and $d \in \mathbb{Z}$ with $\chi_k(d) = -1$, then $$S_{n,k}(d,p) \equiv 0 \pmod{p}.$$
\end{thm}
\begin{rem}
 For $k=2$, we obtain \cite[Theorem $1.1$]{Ren1} from Theorem \ref{thm1}.
\end{rem}
\begin{thm}\label{thm2} Let $p$ be an odd prime and $k\in\mathbb{N}$ with $2\leq k\leq \frac{p-1}{2}$ such that $p>2k+1$ and $p\equiv1$ $($mod $2k)$. Suppose $n \in \{\frac{p-1}{k}+1, \frac{p-1}{k}+2, \cdots, \frac{2(p-1)}{k}-1\}$ is odd and $d\in\mathbb{Z}$ with $\chi_k(d)=1$. 
    \begin{itemize}
        \item[$(a)$] If $k$ is even, then $\displaystyle\left(\frac{S_{n,k}(d,p)}{p}\right) \ne -1.$
        \item[$(b)$]  If $k$ is odd, then $\displaystyle\left(\frac{S_{n,k}(d,p)}{p}\right) \ne
        \begin{cases}
          -1,& \text{if}\ p \equiv 1 \pmod{4k}; \\
          \left(\frac{d}{p}\right), & \text{if}\ p \equiv 2k+1 \pmod{4k}. 
        \end{cases}$
    \end{itemize}
\end{thm}
\begin{rem}
    For $k=2$, Theorem \ref{thm2} readily provides \cite[Theorem 1.2]{Ren1}.
\end{rem}
In \cite{Sun2}, Sun showed that $$S_{m+\frac{p-1}{2},2}(-1,p)\equiv0 \pmod{p}$$ when $m$ is even and $p\equiv3$ $($mod $4)$, and $S_{m+\frac{p-1}{2},2}(-1,p)$ is an integer square modulo $p$ when $m$ is odd and $p\equiv1$ $($mod $4)$. In addition, Sun \cite{Sun2} also posed a number of conjectures related to the determinant $S_{m+\frac{p-1}{2},2}(-1,p)$.
\begin{conj}\cite[Conjecture $6.3$]{Sun2}\label{conj01}
    For any prime $p \equiv 1 \pmod{4},$ we have 
    \begin{equation*}
        \left(\frac{\sqrt{S_{1+\frac{p-1}{2},2}(-1,p)}}{p}\right) = (-1)^{|\{0 < k < \frac{p}{4}:\left(\frac{k}{p}\right)=-1\}|}\left(\frac{p}{3}\right).
    \end{equation*}
\end{conj}
\begin{conj}\cite[Conjecture $6.4$]{Sun2}\label{conj02}
    For any prime $p \equiv 1 \pmod{4},$ we have 
    \begin{equation*}
        \left(\frac{\sqrt{S_{3+\frac{p-1}{2},2}(-1,p)}}{p}\right) = (-1)^{|\{0 < k < \frac{p}{4}:\left(\frac{k}{p}\right)=-1\}|}\left(\frac{p}{4+(-1)^{\frac{p-1}{4}}}\right).
    \end{equation*}
\end{conj}
\begin{conj}\cite[Conjecture $6.5$]{Sun2}\label{conj03}
    For any positive odd integer $m,$ the set 
    \begin{equation*}
        E_2(m) = \{p : p\ \text{is a prime with}\ 4\ |\ p-1\ \text{and}\ p\ |\ S_{m+\frac{p-1}{2},2}(-1,p)\} 
    \end{equation*}
    is finite. In particular,
    \begin{align*}
       & E_2(5) = \{29\},\ E_2(7) = \{13,53\}, E_2(9) = \{13,17,29\}, \\ 
       & E_2(11) = \{17,29\}\ \text{and}\ E_2(13) = \{17,109,401\}.
    \end{align*}
    \begin{rem}
        Sun \cite{Sun2} provided the list of $E_2(m)$ based on his calculations for primes $< 1000.$
    \end{rem}
\end{conj}
If $m$ is odd and $p\equiv1$ $($mod $2k)$, then it is easy to see that $S_{m+\frac{p-1}{k},k}(-1,p)$ is a skew-symmetric matrix of even order, and hence 
$S_{m+\frac{p-1}{k},k}(-1,p)$ is an integral square. For certain values of $m$, we now investigate residue properties of $\sqrt{S_{m+\frac{p-1}{k},k}(-1,p)}$. For $a, e \in \mathbb{N},$ we define the $e$- factorial, denoted by $a!_{(e)}$, as
\begin{equation*}
   a!_{\left(e \right)} =
   \begin{cases}
       a \cdot (a - e)!_{\left(e \right)} , &  \text{if}\ a > e; \\
       a, & \text{if}\ 1 \leq a \leq e.
   \end{cases}
\end{equation*}
In particular, we denote $a!_{(2)}$ by $a!!$
\begin{thm}\label{thm3} Let $p$ be an odd prime and $k\in \mathbb{N}$ with $2 \leq k \leq \frac{p-1}{2}$ such that $p \equiv 1 \pmod{2k}$. Suppose $\alpha_i$ are distinct $k$-th residues modulo $p$ and $T(\frac{p-1}{k}) := \displaystyle\prod_{1 \leq i < j \leq \frac{p-1}{k}}(\alpha_i - \alpha_j).$
    \begin{itemize}
        \item[$(a)$]  If $p \equiv 1 \pmod{4k},$ then $$\left(\frac{\sqrt{S_{1+\frac{p-1}{k},k}(-1,p)}}{p}\right) = \left(\frac{(k-1)(2k-1)}{p}\right) \left(\frac{\frac{p-1}{k}!!}{p}\right) \left(\frac{T\left(\frac{p-1}{k}\right)}{p}\right).$$
        \item[$(b)$] If $k$ is even and $p \equiv 2k+1 \pmod{4k},$ then
        \begin{equation*}
            \left(\frac{\sqrt{S_{1+\frac{p-1}{k},k}(-1,p)}}{p}\right) = 
            \begin{cases}
                \left(\frac{T(\frac{p-1}{k})}{p}\right), & \text{if}\ p = 2k+1;\\
                 \left(\frac{k(2k-1)}{p}\right) \left(\frac{(\frac{p-1}{k}-1)!!}{p}\right) \left(\frac{T(\frac{p-1}{k})}{p}\right), & \text{otherwise.}
            \end{cases}
        \end{equation*}
    \end{itemize}
\end{thm}
For $k=2$, Theorem \ref{thm3} yields 
\begin{equation}\label{eq1}
    \left(\frac{\sqrt{ S_{1+\frac{p-1}{2},2}(-1,p)}}{p}\right) =
        \begin{cases}
            \left(\frac{3}{p}\right) \left(\frac{\frac{p-1}{2}!!}{p}\right)\left(\frac{T\left(\frac{p-1}{2}\right)}{p}\right), &\ \text{if}\ p \equiv 1 \pmod{8}; \\
            \left(\frac{T\left(\frac{p-1}{2}\right)}{p}\right), &\ \text{if}\  p = 5; \\
             \left(\frac{6}{p}\right) \left(\frac{\frac{p-3}{2}!!}{p}\right)\left(\frac{T\left(\frac{p-1}{2}\right)}{p}\right), &\ \text{if}\ p \equiv 5 \pmod{8} \text{ and } p \ne 5. \\
        \end{cases}
    \end{equation}
From Lemma \ref{Sun1.2}, we note that
 $$\left(\frac{2}{p}\right) =  \left(\frac{\frac{p-1}{2}!}{p}\right) =  \left(\frac{\frac{p-1}{2}!!\cdot \frac{p-3}{2}!!}{p}\right),$$ and hence 
 \begin{equation}
 \label{eq2}
     \left(\frac{\frac{p-1}{2}!!}{p}\right) =\left(\frac{\frac{p-3}{2}!!}{p}\right)\left(\frac{2}{p}\right).
 \end{equation} 
Moreover, Lemma \ref{Sun1.3} and Lemma \ref{Sun1.2} together provide
\begin{equation}
\label{eq3}
\left(\frac{T\left(\frac{p-1}{2}\right)}{p}\right) =\left(\frac{2}{p}\right).
\end{equation}
Plugging \eqref{eq2} and \eqref{eq3} in \eqref{eq1}, and then using Lemma \ref{Sun1.1} and Theorem \ref{Burton1}, we immediately have the following corollary.
\begin{cor}\label{conj1}
    Conjecture \ref{conj01} is true.
\end{cor}
We also deduce another similar result in the following theorem.
\begin{thm}\label{thm4} 
    Let $p$ be an odd primes and $k\in\mathbb{N}$ with $2 \leq k \leq \frac{p-1}{2}$ such that $p \equiv 1 \pmod{2k}$. Suppose $\alpha_i$ are distinct $k$-th residues modulo $p$ and $T\left(\frac{p-1}{k}\right) = \displaystyle\prod_{1 \leq i < j \leq \frac{p-1}{k}}(\alpha_i - \alpha_j).$
    \begin{itemize}
        \item[$(a)$] If $p \equiv 1 \pmod{4k},$ then
        \begin{align*}
           \left(\frac{\sqrt{S_{3+\frac{p-1}{k},k}(-1,p)}}{p}\right)  = &  \left(\frac{k(3k-1)(4k-1)\left(6k^3+(3k-1)(2k-1)(k-1)\right)}{p}\right) \\ & \left(\frac{(\frac{p-1}{k}-1)!!}{p}\right) \left(\frac{T\left(\frac{p-1}{k}\right)}{p}\right).  
        \end{align*}
        \item[$(b)$] If $k$ is even and $p \equiv 2k+1 \pmod{4k},$ then
        \begin{equation*}
            \left(\frac{\sqrt{S_{3+\frac{p-1}{k},k}(-1,p)}}{p}\right) = 
            \begin{cases}
                \left(\frac{T\left(\frac{p-1}{k}\right)}{p}\right), & \text{if}\ p = 2k+1;\\
                  \left(\frac{(k-1)(2k-1)(4k-1)\left(6k^3+(3k-1)(2k-1)(k-1)\right)}{p}\right)\\ \left(\frac{\left(\frac{p-1}{k}\right)!!}{p}\right) \left(\frac{T\left(\frac{p-1}{k}\right)}{p}\right), & \text{otherwise.}
            \end{cases}
        \end{equation*}
    \end{itemize}
\end{thm}
For $k = 2$, Theorem $\ref{thm4}$ implies 
\begin{equation*}
       \left(\frac{\sqrt{S_{3+\frac{p-1}{2},2(-1,p)}}}{p}\right) =
        \begin{cases}
            \left(\frac{10}{p}\right) \left(\frac{\frac{p-3}{2}!!}{p}\right)\left(\frac{T\left(\frac{p-1}{2}\right)}{p}\right), &\ \text{if}\ p \equiv 1 \pmod{8}; \\
            \left(\frac{T\left(\frac{p-1}{2}!!\right)}{p}\right), &\ \text{if}\ p = 5; \\
             \left(\frac{3}{p}\right) \left(\frac{\frac{p-1}{2}!!}{p}\right)\left(\frac{T\left(\frac{p-1}{2}\right)}{p}\right), &\ \text{if}\ p \equiv 5 \pmod{8} \text{ and } p \ne 5.
        \end{cases}
\end{equation*}
Using now \eqref{eq2}, \eqref{eq3}, and Theorem \ref{Burton1}, we have the following corollary.
\begin{cor}\label{conj2}
    Conjecture \ref{conj02} is true.
\end{cor}
From Theorem \ref{thm3}, it is easy to see that if $p \equiv 1 \pmod{2k}$ and $p \equiv 1 \pmod{4}$, then $p \nmid S_{1+\frac{p-1}{k},k}(-1,p)$ for any $2\leq k\leq \frac{p-1}{2}$. In a similar way, Theorem \ref{thm4} implies that if $p \equiv 1 \pmod{2k}$ and $p \equiv 1 \pmod{4}$, then $$p\ |\ S_{3+\frac{p-1}{k},k}(-1,p)$$ if and only if $$p\ |\ 6k^3+(3k-1)(2k-1)(k-1).$$ For a fixed $k$ with $2 \leq k \leq \frac{p-1}{2}$, note that $6k^3+(3k-1)((2k-1)(k-1)$ is a non-zero integer, and hence there are finite number of primes $p$ such that $p\ |\  S_{3+\frac{p-1}{k},k}(-1,p).$ We now move our attention to the general determinant $S_{m+\frac{p-1}{k},k}(-1,p)$, and prove a similar result in the following theorem.
\begin{thm}\label{thm5}
Let $p$ be an odd prime and $k\in\mathbb{N}$ with $2 \leq k \leq \frac{p-1}{2}$. For any odd positive integer $m$, the set 
$$E_k(m) = \{ p : p \text{ is a prime with }2k\ |\ (p-1) \text{ and } p\ |\ S_{m+\frac{p-1}{k},k}(-1,p)\}$$
is finite.
\par Moreover, $E_k(1) = \phi$ and if $m \geq 3$ and $p > km+1$, then $p \in E_k(m)$ if and only if $p \equiv 1 \pmod{2k}$ and 
$$p\ |\ \Big((km)!_{(k)} + (km-1)!_{(k)}\Big)$$ or $$p\ |\ \Bigg(\frac{(km-kl)!_{(k)}}{(kl)!_{(k)}} + \frac{(km-kl-1)!_{(k)}}{(kl-1)!_{(k)}}\Bigg)$$ for $l = 1, \cdots, \frac{m-1}{2}$.  
\end{thm} 
Putting $k=2$ in Theorem \ref{thm5}, we obtain the following immediate corollary.
\begin{cor}\label{conj3}
    Conjecture \ref{conj03} is true.
\end{cor}
We now give an application of Theorem \ref{thm5} to obtain the set $E_2(13).$ Note that
 \begin{align*}
     (2m)!!+(2m-1)!! & = 26!! + 25!! =  58917607974225 = 3^6 \cdot 5^2 \cdot 7 \cdot 11 \cdot 13 \cdot 109 \cdot 29629,
 \end{align*}
 which gives $109, 29629\in E_2(13)$ due to Theorem \ref{thm5}. Moreover, we have 
 \begin{equation*}
     \frac{(2m-2l)!!}{(2l)!!} + \frac{(2m-2l-1)!!}{(2l-1)!!} =
        \begin{cases}
            \frac{24!!}{2!!} + \frac{23!!}{1!!} 
            = 3^6 \cdot 5^2 \cdot 7 \cdot 11 \cdot 924397, &\ \text{if}\ l=1; \\
            \frac{22!!}{4!!} + \frac{21!!}{3!!}
            = 3^5 \cdot 5^2 \cdot 7 \cdot 11 \cdot 31643, &\ \text{if}\ l= 2;\\
         \frac{20!!}{6!!} + \frac{19!!}{5!!}= 3^4 \cdot 5 \cdot 7 \cdot 42703, &\ \text{if}\ l=3; \\
            \frac{18!!}{8!!} + \frac{17!!}{7!!}= 3^4 \cdot 5^2 \cdot 401, &\ \text{if}\ l=4; \\
            \frac{16!!}{10!!} + \frac{15!!}{9!!} = 3^3 \cdot 179, &\ \text{if}\ l=5; \\
            \frac{14!!}{12!!} + \frac{13!!}{11!!} = 3^3, &\ \text{if}\ l=6. \\
        \end{cases}
\end{equation*}
As a result, $401, 924397\in E_2(13)$. It is easy to calculate that $p=17$ is the only prime less than $27$ such that $p\ |\ S_{13+\frac{p-1}{2}},2(-1,p)$. Combining all these together, we obtain $E_2(13) = \{17,109,401,29629, 924397\}$.
\section{Preliminaries}
In this section, we recall certain basic results that shall be used to prove our results. We begin with the Quadratic Reciprocity Law.
\begin{thm}\cite[Theorem 9.9]{Burton1}
\label{Burton1}
    If $p$ and $q$ are distinct odd primes, then $$\left(\frac{p}{q}\right) \left(\frac{q}{p}\right) = (-1)^{\frac{p-1}{2}\cdot\frac{q-1}{2}}.$$
\end{thm}
The following result on determinant is known from \cite{Krattenthaler1}.
\begin{lem}\cite[Lemma $10$]{Krattenthaler1}
\label{Krattenthaler1}
    Let $R$ be a commutative ring with identity and let $P(x) = \displaystyle\sum_{i=0}^{n-1}a_ix^i \in R[x].$ 
    Then 
    \begin{equation*}
        \text{det}[P(X_iY_j)]_{1 \leq i,j \leq n} = a_0a_1\cdots a_{n-1} \displaystyle\prod_{1 \leq i< j \leq n}(X_i - X_j)(Y_i - Y_j).
    \end{equation*}
\end{lem}
In \cite{Sun1}, Sun evaluated the following Legendre symbol values of certain factorials.
\begin{lem}\cite[Lemma $2.3$]{Sun1}
\label{Sun1.2}
    Let $p \equiv 1 \pmod{4}$ be a prime, then
    \begin{equation*}
    \left( \frac{\frac{p-1}{2}!}{p}\right) = \left(\frac{2}{p}\right).
    \end{equation*}
\end{lem}
\begin{lem}\cite[Lemma $3.2$]{Sun1}
\label{Sun1.1}
    Let $p$ be an odd prime. If $ p \equiv 1 \pmod{4}$,\ then 
    \begin{equation*}
    \label{eq 3.1}
         \left(\frac{\frac{p-3}{2}!!}{p}\right) = (-1)^{|\{0<k<\frac{p}{4}:\left(\frac{k}{p}\right)=-1\}|}.
    \end{equation*}
\end{lem}
We also restate a result of Sun from \cite{Sun3}.
\begin{lem}\cite[$(1.5)$]{Sun3}\label{Sun1.3} For an odd prime $p$, we have
\begin{equation*}
    \displaystyle\prod_{1 \leq i < j \leq \frac{p-1}{2}}(j^2-i^2) \equiv 
     \begin{cases}
      -\frac{p-1}{2}!\pmod{p}, & \text{if}\ p\equiv 1\pmod{4} \\  
       1 \pmod{p}, & \text{if}\ p\equiv 3\pmod{4}
    \end{cases}
\end{equation*}
\end{lem}
Finally, we deduce a congruence relation for $S_{m,k}(d,p)$ modulo $p$ using a method introduced by Ren and Sun \cite{Ren1}.
\begin{lem}\label{thm0}
    Let $p$ be an odd prime and $k\in\mathbb{N}$ with $2\leq k\leq \frac{p-1}{2}$ such that $p\equiv1$ $($mod $k)$. Suppose $m,d\in\mathbb{Z}$ such that $p\nmid d$, $\chi_k(d)=\pm 1$, and $m \in \{\frac{p-1}{k}+1, \frac{p-1}{k}+2, \cdots, \frac{2(p-1)}{k}-1\}$. If $p  > 2k+1$, then
    $$S_{m,k}(d,p) \equiv a^2_{m,k}(d,p)\ b_{m,k}(d,p) \pmod{p},$$ 
    where 
    \begin{align*}
        a_{m,k}(d,p)  = & \displaystyle\prod_{l=0}^{\lfloor \frac{(k(m-1)-p+1}{2k} \rfloor} \left[ \binom{m}{l} + \chi_k(d) \binom{m}{m-\frac{p-1}{k}-l}\right] \displaystyle\prod_{0 \leq l < \frac{p-1}{k}-1- \lfloor \frac{m}{2} \rfloor} \binom{m}{m-\frac{p-1}{k}+1+l} \\ & \times \displaystyle\prod_{1 \leq i < j \leq \frac{p-1}{k}}(\alpha_i - \alpha_j)
    \end{align*}
    and 
    \begin{equation*}
        b_{m,k}(d,p) = 
        \begin{cases}
            \displaystyle \chi_{2k}(d) (-1)^{\frac{p-1}{2k}-1}(1+\chi_k(d))\displaystyle\binom{m}{\frac{km-p+1}{2k}}\binom{m}{\frac{m}{2}}, & \text{if}\ 2\ |\ m \ \text{and}\ p \equiv 1 \pmod{2k}, \\
            \displaystyle(\chi_k(d))^{\frac{m}{2}} (-1)^{\frac{p-k-1}{2k}} \displaystyle\binom{m}{\frac{m}{2}}, & \text{if}\ 2\ |\ m \ \text{and}\ p \equiv k+1 \pmod{2k}, \\
            \displaystyle (\chi_{2k}(d))^m (-1)^{\frac{p-1}{2k}};   & \text{if}\ 2 \nmid m \ \text{and}\ p \equiv 1 \pmod{2k}, \\
            \displaystyle (-1)^{\frac{p-k-1}{2k}} (1+\chi_k(d)) \binom{m}{\frac{km-p+1}{2k}},  & \text{if}\ 2 \nmid m \ \text{and}\ p \equiv k+1 \pmod{2k}. 
        \end{cases}
    \end{equation*}
\end{lem}
\begin{proof}
Let $g$ be a primitive root modulo $p$, then $\{g^{ki}\}_{i=1}^{\frac{p-1}{k}}$ is a permutation of all distinct $k$-th power residues modulo $p$. Therefore,
$$\displaystyle\prod_{i=1}^{\frac{p-1}{k}} \alpha_i \equiv \displaystyle\prod_{i = 1}^{\frac{p-1}{k}} g^{ki} = \{g^\frac{p-1}{2}\}^{\frac{p-1}{k}+1} \equiv  (-1)^{\frac{p-1}{k}+1}  \pmod{p}.$$
Using this, we have
    \begin{align}\label{eqs1}
        S_{m,k}(d,p) & = \displaystyle\prod_{j=1}^{\frac{p-1}{k}} \alpha_j^m \left|\left(\frac{\alpha_i}{\alpha_j}+d\right)^m\right|_{1 \leq i,j \leq \frac{p-1}{k}} \equiv  (-1)^{m(\frac{p-1}{k}+1)}  \left|\left(\frac{\alpha_i}{\alpha_j}+d\right)^m\right|_{1 \leq i,j \leq \frac{p-1}{k}} \pmod{p}.
    \end{align}
Note that
    \begin{align*}
       \left(\frac{\alpha_i}{\alpha_j}+d\right)^m 
       & = \displaystyle\sum_{l=0}^{m-\frac{p-1}{k}}  \binom{m}{l}d^{m-l}\left(\frac{\alpha_i}{\alpha_j}\right)^l + \displaystyle\sum_{m-\frac{p-1}{k}+1\leq l < \frac{p-1}{k}}\binom{m}{l}d^{m-l}\left(\frac{\alpha_i}{\alpha_j}\right)^l \\ & \hspace{3cm} +  \displaystyle\sum_{l = \frac{p-1}{k}}^{m}\binom{m}{l}d^{m-l}\left(\frac{\alpha_i}{\alpha_j}\right)^l \\ & \equiv \displaystyle\sum_{l=0}^{m-\frac{p-1}{k}} \Bigg[\binom{m}{l}+\binom{m}{\frac{p-1}{k}+l}d^{-\frac{p-1}{k}}\Bigg]d^{m-l}\left(\frac{\alpha_i}{\alpha_j}\right)^l \\ &  \hspace{3cm}+  \displaystyle\sum_{m-\frac{p-1}{k}+1\leq l < \frac{p-1}{k}}\binom{m}{l}d^{m-l}\left(\frac{\alpha_i}{\alpha_j}\right)^l \pmod{p} \\
       &\equiv  \displaystyle\sum_{l = 0}^{m-\frac{p-1}{k}} \Bigg[\binom{m}{l} + \binom{m}{m-\frac{p-1}{k}-l}d^{-\frac{p-1}{k}}\Bigg]d^{m-l}\left(\frac{\alpha_i}{\alpha_j}\right)^l  \\ &\hspace{2cm}+ \displaystyle\sum_{0 \leq l < \frac{2(p-1)}{k}-m-1} \binom{m}{l+m-\frac{p-1}{k}+1} d^{\frac{p-1}{k}-1-l}\left(\frac{\alpha_i}{\alpha_j}\right)^{l+m-\frac{p-1}{k}+1} \pmod{p}.
    \end{align*}
    We now use Lemma \ref{Krattenthaler1} to obtain
    \begin{align*}
       \left|\left(\frac{\alpha_i}{\alpha_j}+d\right)^m\right|_{1 \leq i,j \leq \frac{p-1}{k}}  \equiv &  \displaystyle\prod_{l=0}^{m-\frac{p-1}{k}}\Bigg[\binom{m}{l}+d^{-\frac{p-1}{k}}\binom{m}{m-\frac{p-1}{k}-l}\Bigg]d^{m-l}\\ & \displaystyle\prod_{0\leq l < \frac{2(p-1)}{k}-m-1}\binom{m}{l+m-\frac{p-1}{k}+1}d^{\frac{p-1}{k}-1-l} \displaystyle\prod_{1 \leq i < j \leq \frac{p-1}{k}} (\alpha_i-\alpha_j)\left(\frac{1}{\alpha_i}-\frac{1}{\alpha_j}\right)\pmod{p}.
        \end{align*}
       Since
    \begin{equation*}
        \displaystyle\prod_{1 \leq i < j \leq \frac{p-1}{k}}\alpha_i \alpha_j = \displaystyle\prod_{1 \leq j \leq \frac{p-1}{k}}\alpha_j^{\frac{p-1}{k}-1} \equiv (-1)^{\frac{p-1}{k}+1} \pmod{p},
    \end{equation*} we have 
    \begin{align*}
        \displaystyle\prod_{1 \leq i < j \leq \frac{p-1}{k}}\left(\alpha_i - \alpha_j\right)\left(\frac{1}{\alpha_i} - \frac{1}{\alpha_j}\right) &  = (-1)^{\binom{\frac{p-1}{k}}{2}} \displaystyle\prod_{1 \leq i < j \leq \frac{p-1}{k}} \frac{(\alpha_i-\alpha_j)^2}{\alpha_i\alpha_j}\\ & \equiv (-1)^{\lfloor \frac{p-k-1}{2k} \rfloor} \displaystyle\prod_{1 \leq i < j \leq \frac{p-1}{k}} (\alpha_i-\alpha_j)^2 \pmod{p}.
    \end{align*}
    As a result,
        \begin{align*}
        \left|\left(\frac{\alpha_i}{\alpha_j}+d\right)^m\right|_{1 \leq i,j \leq \frac{p-1}{k}}   \equiv & 
         (-1)^{\lfloor \frac{p-k-1}{2k} \rfloor} d^{\frac{\frac{p-1}{k}(2m-\frac{p-1}{k}+1)}{2}}\displaystyle\prod_{l=0}^{m-\frac{p-1}{k}}\Bigg[\binom{m}{l}+d^{-\frac{p-1}{k}}\binom{m}{m-\frac{p-1}{k}-l}\Bigg]\notag\\ &\displaystyle\prod_{0\leq l < \frac{2(p-1)}{k}-m-1}\binom{m}{l+m-\frac{p-1}{k}+1} \displaystyle\prod_{1 \leq i < j \leq \frac{p-1}{k}} (\alpha_i-\alpha_j)^2 \pmod{p}.
    \end{align*}
    Using this in \eqref{eqs1}, we obtain
     \begin{align}\label{eqs2}
        S_{m,k}(d,p) \equiv&  (-1)^{m(\frac{p-1}{k}+1)+\lfloor \frac{p-k-1}{2k} \rfloor} d^{\frac{\frac{p-1}{k}(2m-\frac{p-1}{k}+1)}{2}}\displaystyle\prod_{l=0}^{m-\frac{p-1}{k}}\Bigg[\binom{m}{l}+d^{-\frac{p-1}{k}}\binom{m}{m-\frac{p-1}{k}-l}\Bigg]\notag\\ & \displaystyle\prod_{0\leq l < \frac{2(p-1)}{k}-m-1}\binom{m}{l+m-\frac{p-1}{k}+1} \displaystyle\prod_{1 \leq i < j \leq \frac{p-1}{k}} (\alpha_i-\alpha_j)^2 \pmod{p}.
    \end{align}
It is easy to see that
    \begin{equation}\label{eqs3}
       \displaystyle\prod_{0\leq l < \frac{2(p-1)}{k}-m-1} \binom{m}{l+m-\frac{p-1}{k}+1} = 
        \begin{cases}
            \displaystyle\binom{m}{\frac{m}{2}} \displaystyle\prod_{0 \leq l < \frac{p-1}{k}-1-\frac{m}{2}}  \binom{m}{l+m-\frac{p-1}{k}+1}^2, & \text{if}\ 2\ |\ m; \\
            \displaystyle\prod_{0 \leq l < \frac{p-1}{k}-1-\frac{m-1}{2}}  \binom{m}{l+m-\frac{p-1}{k}+1}^2, & \text{if}\ 2\ \nmid m. \\
        \end{cases}
    \end{equation}
\textbf{Case I:} Let $m\equiv\frac{p-1}{k}$ $($mod $2)$. Noting that $\chi_k(d)=d^{\frac{p-1}{k}}\equiv\pm 1$ $($mod $p)$, we have   
    \begin{align}\label{eqs4}
        \displaystyle\prod_{l=0}^{m-\frac{p-1}{k}} \Bigg\{\binom{m}{l}&+ d^{-\frac{p-1}{k}} \binom{m}{m-\frac{p-1}{k}-l}\Bigg\}\notag\\  =&  (1+d^{-\frac{p-1}{k}})\binom{m}{\frac{m-\frac{p-1}{k}}{2}}\displaystyle\prod_{l=0}^{\frac{m-\frac{p-1}{k}}{2}-1}\Bigg[\binom{m}{l}+ d^{-\frac{p-1}{k}} \binom{m}{m-\frac{p-1}{k}-l}\Bigg] \notag\\ & \hspace{2cm}\displaystyle\prod_{l=\frac{m-\frac{p-1}{k}}{2}+1}^{m-\frac{p-1}{k}}\Bigg[\binom{m}{l}+ d^{-\frac{p-1}{k}} \binom{m}{m-\frac{p-1}{k}-l}\Bigg]  \notag\\  \equiv & (1+d^{-\frac{p-1}{k}})(d^{-\frac{p-1}{k}})^{\frac{km-p+1}{2k}}\binom{m}{\frac{m-\frac{p-1}{k}}{2}}\notag\\
       &\hspace{2cm}\displaystyle\prod_{l=0}^{\frac{m-\frac{p-1}{k}}{2}-1}\Bigg[\binom{m}{l}+ d^{-\frac{p-1}{k}} \binom{m}{m-\frac{p-1}{k}-l}\Bigg]^2 \pmod{p}.
    \end{align}
Using \eqref{eqs3} and \eqref{eqs4} in \eqref{eqs2}, and then using the facts that 
\begin{align*}
    d^{\frac{\frac{p-1}{k}(2m-\frac{p-1}{k}+1)}{2}}(d^{-\frac{p-1}{k}})^{\frac{km-p+1}{2k}} &\equiv  d^{\frac{(m+1)(p-1)}{2k}} \pmod{p}\\
    &\equiv 
        \begin{cases}
            \left(\chi_{2k}(d)\right)^{m+1} ~~\pmod{p}, & \text{if both $m,\frac{p-1}{k}$ are even;}\\
           \left(\chi_k(d)\right)^{\frac{m+1}{2}}~~\pmod{p}, & \text{if both $m,\frac{p-1}{k}$ are odd;}\\
        \end{cases}
    \end{align*}
    and  
    \begin{equation*}
        (-1)^{m(\frac{p-1}{k}+1) + \lfloor \frac{p-k-1}{2k} \rfloor} =
        \begin{cases}
         (-1)^{\frac{p-1}{2k}-1}; & \ \text{if both $m,\frac{p-1}{k}$ are even;} \\
        (-1)^{\frac{p-k-1}{2k}}; & \ \text{if both $m,\frac{p-1}{k}$ are odd;} \\
        \end{cases}
    \end{equation*}
    we obtain the desired result.\\\\
\textbf{Case II}: Let $m\not\equiv \frac{p-1}{k}$ $($mod $2)$. In this case, we have
    \begin{align}\label{eqs5}
         \displaystyle\prod_{l=0}^{m-\frac{p-1}{k}} \Bigg\{\binom{m}{l}&+ d^{-\frac{p-1}{k}} \binom{m}{m-\frac{p-1}{k}-l}\Bigg\}\notag \\ = &  \displaystyle\prod_{l=0}^{\frac{m-\frac{p-1}{k}-1}{2}}\Bigg[\binom{m}{l}+ d^{-\frac{p-1}{k}} \binom{m}{m-\frac{p-1}{k}-l}\Bigg]\notag\\
         &\hspace{2cm}\displaystyle\prod_{l=\frac{m-\frac{p-1}{k}+1}{2}}^{m-\frac{p-1}{k}}\Bigg[\binom{m}{l}+ d^{-\frac{p-1}{k}} \binom{m}{m-\frac{p-1}{k}-l}\Bigg] \notag\\  \equiv & (d^{-\frac{p-1}{k}})^{\frac{k(m+1)-p+1}{2k}} \displaystyle\prod_{l=0}^{\frac{m-\frac{p-1}{k}-1}{2}}\Bigg[\binom{m}{l}+ d^{-\frac{p-1}{k}} \binom{m}{m-\frac{p-1}{k}-l}\Bigg]^2 \pmod{p}.
    \end{align}
Putting \eqref{eqs3} and \eqref{eqs5} in \eqref{eqs2}, we use the facts 
    \begin{align*}
        d^{\frac{\frac{p-1}{k}(2m-\frac{p-1}{k}+1)}{2}}(d^{-\frac{p-1}{k}})^{\frac{k(m+1)-p+1}{2k}} &\equiv 
            d^{\frac{m(p-1)}{2k}} \pmod{p}\\
    &\equiv 
       \begin{cases}
           \left(\chi_k(d)\right)^{\frac{m}{2}}   ~~\pmod{p}, & \text{if $m$ is even and $\frac{p-1}{k}$ is odd;}\\
        \left(\chi_{2k}(d)\right)^{m}    ~~\pmod{p}, & \text{if $m$ is odd and $\frac{p-1}{k}$ is even;}\\
        \end{cases}
    \end{align*}
and
    \begin{equation*}
        (-1)^{m(\frac{p-1}{k}+1) + \lfloor \frac{p-k-1}{2k} \rfloor} =
        \begin{cases}
         (-1)^{\frac{p-k-1}{2k}}, & \ \text{if $m$ is even and $\frac{p-1}{k}$ is odd}; \\
         (-1)^{\frac{p-1}{2k}},  & \ \text{if $m$ is odd and $\frac{p-1}{k}$ is even}; \\
        \end{cases}
    \end{equation*}
    to complete the proof of the lemma.
\end{proof}
\section{Proof of the main results}  
In this section, we give proof of our main results. 
\begin{pf} \textbf{\ref{thm1}.}
We have $n \equiv \frac{p-1}{k} \pmod{2}$. Therefore, both $n$ and $\frac{p-1}{k}$ are even or odd. Since $\chi_k(d)=-1$, we must have from Lemma \ref{thm0} that $b_{n,k}(d,p)=0$, and hence we complete the proof of the theorem.
  \end{pf} 
 \begin{pf} \textbf{\ref{thm2}.}
From Lemma \ref{thm0}, we have 
    \begin{equation*}
        \left(\frac{S_{n,k}(d,p)}{p}\right) = \left(\frac{a_{n,k}(d,p)}{p}\right)^2\ \left(\frac{b_{n,k}(d,p)}{p}\right).
    \end{equation*}
Since $n$ is odd and $p\equiv1 \pmod{2k}$, Lemma \ref{thm0} provides
    \begin{equation*}
        \left(\frac{b_{n,k}(d,p)}{p}\right) = \left(\frac{d}{p}\right)^{\frac{p-1}{2k}}\ \left(\frac{\chi_k(d)}{p}\right)^{\frac{n-1}{2}}\ \left(\frac{-1}{p}\right)^{\frac{p-1}{2k}} = \left(\frac{-d}{p}\right)^{\frac{p-1}{2k}},
    \end{equation*}
    and hence 
    \begin{equation}\label{eqb1}
        \left(\frac{S_{n,k}(d,p)}{p}\right) = \left(\frac{a_{n,k}(d,p)}{p}\right)^2\left(\frac{-d}{p}\right)^{\frac{p-1}{2k}}.
    \end{equation}
 $(a)$ Let $k$ be even. Since $$\chi_k(d) \equiv d^{\frac{p-1}{k}} \equiv 1 \pmod{p},$$ we have $$d^{\frac{p-1}{2}} = (d^{\frac{p-1}{k}})^{\frac{k}{2}} \equiv 1 \pmod{p}.$$ As a result, $$\left(-\frac{d}{p}\right)^{\frac{p-1}{2k}} = 1,$$ and hence \eqref{eqb1} yields $$\left(\frac{S_{n,k}(d,p)}{p}\right) = \left(\frac{a_{n,k}(d,p)}{p}\right)^2 = \begin{cases}
         1, & \ \text{if $a_{n,k}(d,p)\neq0$}; \\
         0,  & \ \text{if $a_{n,k}(d,p)=0$}. \\
        \end{cases}$$
        Thus we obtain the desired result.\\\\
 $(b)$  If $k$ is odd, then it is easy to see that
    \begin{equation*}
         \left(-\frac{d}{p}\right)^{\frac{p-1}{2k}} = 
        \begin{cases}
            1, & \text{if}\ p \equiv 1 \pmod{4k} \\
            -\left(\frac{d}{p}\right), & \text{if}\ p \equiv 2k+1 \pmod{4k}.
        \end{cases}
    \end{equation*}
   Using this in \eqref{eqb1} , and then noting that
    $$\left(\frac{S_{n,k}(d,p)}{p}\right) = 0$$ if and only if $$\left(\frac{a_{n,k}(d,p)}{p}\right) = 0,$$
    we complete the proof of the theorem.
    \end{pf} 
\begin{pf} \textbf{\ref{thm3}.} We first note that $p\equiv1 \pmod{4}$ in each case, and hence $\left(\frac{-1}{p}\right)=1$.
If $p = 2k+1,$ then  
\begin{equation}\label{eq6}
   S_{1+\frac{p-1}{k},k}(-1,p) = S_{3,k}(-1,p) =  \begin{vmatrix}
     0 & (\alpha_1 - \alpha_2)^3 \\ 
     (\alpha_2 - \alpha_1)^3 & 0\\
\end{vmatrix}= (\alpha_1 - \alpha_2)^6.
\end{equation}
As a result,
\begin{equation*}
    \left(\frac{\sqrt{S_{1+\frac{p-1}{k},k}(-1,p)}}{p}\right) =\left(\frac{\pm(\alpha_1 - \alpha_2)^3}{p}\right)= \left(\frac{\alpha_1 - \alpha_2}{p}\right) = \left(\frac{T(\frac{p-1}{k})}{p}\right).
\end{equation*}
Suppose $p > 2k+1$ such that $p\equiv1$ $($mod $2k)$. Note that $\frac{p-1}{k}$ is even and $\frac{p-1}{k}+1$ is odd, and hence Lemma \ref{thm0} provides
 $$S_{\frac{p-1}{k}+1,k}(-1,p)  \equiv  a_{\frac{p-1}{k}+1,k}^2(-1,p) b_{\frac{p-1}{k}+1,k}(-1,p) \pmod{p},$$ where
\begin{align*}
    a_{\frac{p-1}{k}+1,k}(-1,p) & = \left(\frac{p-1}{k}+2\right)\displaystyle\prod_{l=0}^{\frac{p-1}{2k}-2}\binom{\frac{p-1}{k}+1}{2+l}  \displaystyle\prod_{1 \leq i < j \leq \frac{p-1}{k}}(\alpha_i - \alpha_j)
\end{align*}
and 
\begin{equation*}
    b_{\frac{p-1}{k}+1,k}(-1,p) = 1.
\end{equation*} 
Therefore, 
\begin{align*}
  \sqrt{S_{\frac{p-1}{k}+1,k}(-1,p)} \equiv a_{\frac{p-1}{k}+1,k}(-1,p) \pmod{p}\ \text{or}\ -a_{\frac{p-1}{k}+1,k}(-1,p) \pmod{p}.
\end{align*}
Since $\left(\frac{-1}{p}\right)=1$, we must have 
\begin{equation}\label{eqs1k}
    \left(\frac{ \sqrt{S_{\frac{p-1}{k}+1,k}(-1,p)}}{p}\right) 
    = \left(\frac{\left(\frac{p-1}{k}+2\right)\displaystyle\prod_{l=0}^{\frac{p-1}{2k}-2}\binom{\frac{p-1}{k}+1}{2+l}}{p}\right)\left(\frac{\displaystyle\prod_{1 \leq i < j \leq \frac{p-1}{k}}(\alpha_i - \alpha_j)}{p}\right).
\end{equation}
Note that
\begin{equation}\label{eq4}
     \displaystyle\prod_{l=0}^{\frac{p-1}{2k}-2} \binom{\frac{p-1}{k}+1}{2+l} = \left(\frac{p-1}{k}+1\right)^{\frac{p-1}{2k}-1} \displaystyle\prod_{r=1}^{\frac{p-1}{2k}-1}\frac{\left(\frac{p-1}{k}-r+1\right)^{\frac{p-1}{2k}-r}}{\left(\frac{p-1}{2k}-r+1\right)^{r}} ,
\end{equation}
\textbf{Case I:} Let $p \equiv 1 \pmod{4k}$, then \eqref{eq4} yields
\begin{align*}
  \left(\frac{\displaystyle\prod_{l=0}^{\frac{p-1}{2k}-2} \binom{\frac{p-1}{k}+1}{2+l}}{p}\right) = \left(\frac{\left(\frac{p-1}{k}\right)!!}{p}\right) \left(\frac{\frac{p-1}{k}+1}{p}\right).
\end{align*}
Using this in \eqref{eqs1k}, we obtain
\begin{align*}
  \left(\frac{\sqrt{S_{\frac{p-1}{k}+1,k}(-1,p)}}{p}\right) & = \left(\frac{(\frac{p-1}{k}+1)(\frac{p-1}{k}+2)}{p}\right) \left(\frac{\left(\frac{p-1}{k}\right)!!}{p}\right) \left(\frac{\displaystyle\prod_{1 \leq i < j \leq \frac{p-1}{k}}(\alpha_i - \alpha_j)}{p}\right).
\end{align*}
Hence the result follows.\\\\
\textbf{Case II:}
    Let $p \equiv 2k+1 \pmod{4k}$. In this case, we have from \eqref{eq4} that
    \begin{align*}
       \left(\frac{\displaystyle\prod_{l=0}^{\frac{p-1}{2k}-2} \binom{\frac{p-1}{k}+1}{2+l}}{p}\right) = \left(\frac{(\frac{p-1}{k}-1)!!}{p}\right).
    \end{align*}
    Using this in \eqref{eqs1k}, we obtain
    \begin{align*}
        \left(\frac{\sqrt{S_{\frac{p-1}{k}+1,k}(-1,p)}}{p}\right) & = \left(\frac{(\frac{p-1}{k}+2)}{p}\right) \left(\frac{(\frac{p-1}{k}-1)!!}{p}\right) \left(\frac{\displaystyle\prod_{1 \leq i < j \leq \frac{p-1}{k}}(\alpha_i - \alpha_j)}{p}\right),
    \end{align*}
    completing proof of the theorem.
\end{pf}
\begin{pf} \textbf{\ref{thm4}.} It is easy to see that $p\equiv1 \pmod{4}$ in each case, and hence $\left(\frac{-1}{p}\right)=1$.
If $p = 2k+1,$ then 
\begin{equation*}
    S_{5,k}(-1,p) =  \begin{vmatrix}
     0 & (\alpha_1 - \alpha_2)^5 \\ 
     (\alpha_2 - \alpha_1)^5 & 0\\
\end{vmatrix}= (\alpha_1 - \alpha_2)^{10},
\end{equation*}
and hence
\begin{equation*}
    \left(\frac{\sqrt{S_{5,k}(-1,p)}}{p}\right) = \left(\frac{\pm(\alpha_1 - \alpha_2)^5}{p}\right)=\left(\frac{\alpha_1 - \alpha_2}{p}\right).
\end{equation*}
Suppose $p >2k+1$ such that $p \equiv 1 \pmod{2k}$. Noting that $\frac{p-1}{k}$ is even and $\frac{p-1}{k}+3$ is odd, we have from Lemma \ref{thm0} that $$S_{\frac{p-1}{k}+3,k}(-1,p)  \equiv  a_{\frac{p-1}{k}+3,k}^2(-1,p)\ b_{\frac{p-1}{k}+3,k}(-1,p) \pmod{p},$$ where
\begin{align*}
    a_{\frac{p-1}{k}+3,k}(-1,p) & =\displaystyle\prod_{l = 0}^{1} \Bigg[\binom{\frac{p-1}{k}+3}{l} + \binom{\frac{p-1}{k}+3}{3-l}\Bigg] 
    \displaystyle\prod_{0 \leq l < \frac{p-1}{2k}-2}\binom{\frac{p-1}{k}+3}{4+l}\displaystyle\prod_{1 \leq i < j \leq \frac{p-1}{k}}(\alpha_i - \alpha_j) 
\end{align*}
and 
\begin{equation*}
    b_{\frac{p-1}{k}+3,k}(-1,p) = 1.
\end{equation*} 
As a result, we have 
\begin{align*}
   \sqrt{S_{\frac{p-1}{k}+3,k}(-1,p)} \equiv a_{\frac{p-1}{k}+3,k}(-1,p) \ \text{or}\ -a_{\frac{p-1}{k}+3,k}(-1,p) \pmod{p}.
\end{align*}
Since $\left(\frac{-1}{p}\right)=1$, we must have 
\begin{equation*}
    \left(\frac{ \sqrt{S_{\frac{p-1}{k}+3,k}(-1,p)}}{p}\right) = \left(\frac{a_{\frac{p-1}{k}+3,k}(-1,p)}{p}\right).
\end{equation*}
Noting that $$\displaystyle\prod_{l = 0}^{1} \Bigg[\binom{\frac{p-1}{k}+3}{l} + \binom{\frac{p-1}{k}+3}{3-l}\Bigg] \equiv \Bigg[\frac{6k^3+(3k-1)(2k-1)(k-1)}{6k^3}\Bigg]\Bigg[\frac{(3k-1)(4k-1)}{2k^2}\Bigg] \ \pmod{p},$$ we have
\begin{align}\label{eqs3k2}
\left(\frac{ \sqrt{S_{\frac{p-1}{k}+3,k}(-1,p)}}{p}\right)  =  & \left(\frac{3k(3k-1)(4k-1)\{6k^3+(3k-1)(2k-1)(k-1)\}}{p}\right)\notag\\ & \left(\frac{ \displaystyle\prod_{0 \leq l < \frac{p-1}{2k}-2}\binom{\frac{p-1}{k}+3}{4+l}}{p}\right) \left(\frac{\displaystyle\prod_{1 \leq i < j \leq \frac{p-1}{k}}(\alpha_i - \alpha_j)}{p}\right) 
\end{align}
\textbf{Case I:} Let $p\equiv1 \pmod{4k}$. If $p = 4k+1,$ then $$\displaystyle\prod_{0 \leq l < \frac{p-1}{2k}-2}\binom{\frac{p-1}{k}+3}{4+l}=1,$$
and hence the result follows from \eqref{eqs3k2}.
\par On the other hand, if $p>4k+1$, then 
\begin{align*}
     \displaystyle\prod_{l=0}^{\frac{p-1}{2k}-3} \binom{\frac{p-1}{k}+3}{4+l} =& \left(\frac{(\frac{p-1}{k}+3)(\frac{p-1}{k}+2)(\frac{p-1}{k}+1)\frac{p-1}{k}}{4\cdot3\cdot2\cdot1}\right)^{\frac{p-1}{2k}-2} \displaystyle\prod_{r=1}^{\frac{p-1}{2k}-3}\frac{\left(\frac{p-1}{k}-r\right)^{\frac{p-1}{2k}-r-2}}{\left(\frac{p-1}{2k}-r+2\right)^{r}},
\end{align*}
and hence
\begin{align*}
  \left(\frac{\displaystyle\prod_{l=0}^{\frac{p-1}{2k}-3} \binom{\frac{p-1}{k}+3}{4+l}}{p}\right) = \left(\frac{(\frac{p-1}{k}-1)!!}{p}\right) \left(\frac{3}{p}\right).
\end{align*}
Thus we obtain the desired result due to \eqref{eqs3k2}.\\\\
\textbf{Case II:} Let $p \equiv 2k+1 \pmod{4k}$. In this case, we have
\begin{align*}
     \displaystyle\prod_{l=0}^{\frac{p-1}{2k}-3} \binom{\frac{p-1}{k}+3}{4+l} =& \left(\frac{(\frac{p-1}{k}+3)(\frac{p-1}{k}+2)(\frac{p-1}{k}+1)\frac{p-1}{k}}{4\cdot3\cdot2\cdot1}\right)^{\frac{p-1}{2k}-2} \displaystyle\prod_{r=1}^{\frac{p-1}{2k}-3}\frac{\left(\frac{p-1}{k}-r\right)^{\frac{p-1}{2k}-r-2}}{\left(\frac{p-1}{2k}-r+2\right)^{r}},
\end{align*}
 and hence
    \begin{align*}
       \left(\frac{\displaystyle\prod_{l=0}^{\frac{p-1}{2k}-3} \binom{\frac{p-1}{k}+3}{4+l}}{p}\right) & = \left(\frac{\frac{p-1}{k}!!}{p}\right)\left(\frac{3}{p}\right)\left(\frac{(\frac{p-1}{k}+3)(\frac{p-1}{k}+2)(\frac{p-1}{k}+1)}{p}\right) \\ & =  \left(\frac{\frac{p-1}{k}!!}{p}\right)\left(\frac{3}{p}\right)\left(\frac{k(k-1)(2k-1)(3k-1)}{p}\right).
    \end{align*}
Using this in \eqref{eqs3k2}, we complete the proof of the theorem.
\end{pf}  
\begin{pf} \textbf{\ref{thm5}.} We first give a proof of the theorem for $m = 1$. If $p = 2k+1,$ then \eqref{eq6} implies that
$$p \nmid S_{1+\frac{p-1}{k},k}(-1,p),$$ and hence $E_k(1)=\phi$.
\par On the other hand, if $p > 2k+1$, then Lemma \ref{thm0} yields
\begin{equation*}
    S_{1+\frac{p-1}{k},k}(-1,p) \equiv a_{1+\frac{p-1}{k},k}^2(-1,p)\  b_{1+\frac{p-1}{k},k}(-1,p) \pmod{p},
\end{equation*}
where 
\begin{equation*}
     a_{1+\frac{p-1}{k},k}(-1,p) \equiv \left(\frac{2k-1}{k}\right) \ \displaystyle\prod_{l=0}^{\frac{p-1}{2k}-2}\binom{\frac{p-1}{k}+1}{2+l}  \displaystyle\prod_{1 \leq i < j \leq \frac{p-1}{k}}(\alpha_i - \alpha_j) \pmod{p}
\end{equation*}
and 
\begin{equation*}
    b_{m+\frac{p-1}{k},k}(-1,p) = 1.
\end{equation*}
It is easy to see that $\frac{p-1}{k}< 1+\frac{p-1}{k} < \frac{2(p-1)}{k}$. Therefore, $$p \nmid \displaystyle\prod_{l=0}^{\frac{p-1}{2k}-2}\binom{\frac{p-1}{k}+1}{2+l}.$$ As a result, $$p \nmid a_{m+\frac{p-1}{k},k}(-1,p),$$ and hence $$p \ \nmid \  S_{1+\frac{p-1}{k},k}(-1,p),$$ concluding that $E_k(1)=\phi$.
 \par We now assume that $m \geq 3$ is odd and  $p > km+1.$ Note that $\frac{p-1}{k} < m+\frac{p-1}{k} < \frac{2(p-1)}{k}$ and $\frac{p-1}{k}$ is even, and hence Lemma \ref{thm0} provides 
 $$S_{m+\frac{p-1}{k},k}(-1,p)  \equiv  a_{m+\frac{p-1}{k},k}^2(-1,p)\ b_{m+\frac{p-1}{k},k}(-1,p) \pmod{p},$$ where
\begin{align*}
    a_{m+\frac{p-1}{k},k}(-1,p) & = \displaystyle\prod_{l=0}^{\frac{m-1}{2}} \Bigg[\binom{m+\frac{p-1}{k}}{l} + \binom{m+\frac{p-1}{k}}{m-l}\Bigg]  \displaystyle\prod_{0 \leq l < \frac{p-1}{2k}- \frac{m+1}{2}} \binom{m+\frac{p-1}{k}}{m+1+l} \displaystyle\prod_{1 \leq i < j \leq \frac{p-1}{k}}(\alpha_i - \alpha_j) 
\end{align*}
and
\begin{equation*}
    b_{m+\frac{p-1}{k},k}(-1,p) = 1.
\end{equation*}
It is easy to see that
$$p \nmid \displaystyle\prod_{0 \leq l < \frac{p-1}{2k}- \frac{m+1}{2}} \binom{m+\frac{p-1}{k}}{m+1+l} \displaystyle\prod_{1 \leq i < j \leq \frac{p-1}{k}}(\alpha_i - \alpha_j). $$ 
Therefore, we must have that
$$p\ |\ S_{m+\frac{p-1}{k},k}(-1,p)$$ if and only if 
\begin{align}\label{eqif1}
p\ |\ \displaystyle\prod_{l=0}^{\frac{m-1}{2}} \Bigg[\binom{m+\frac{p-1}{k}}{l} + \binom{m+\frac{p-1}{k}}{m-l}\Bigg].
\end{align}
For $l = 0$, note that 
\begin{align}\label{eqif2}
\binom{m+\frac{p-1}{k}}{l} + \binom{m+\frac{p-1}{k}}{m-l} = 1 + \binom{m+\frac{p-1}{k}}{m} \equiv \frac{(km)!_{(k)}+(km-1)!_{(k)}}{(km)!_{(k)}} \pmod{p}.
\end{align}
Again, for $l=1, 2, \ldots, \frac{m-1}{2}$, we have
\begin{align}\label{eqif3}
    \binom{m+\frac{p-1}{k}}{l} + \binom{m+\frac{p-1}{k}}{m-l}   = & \frac{\left(m+\frac{p-1}{k}\right)\left(m+\frac{p-1}{k}-1\right) \cdots \left(m+\frac{p-1}{k}-(l-1)\right)}{l!}\notag \\
    &\hspace{1cm}+ \frac{\left(m+\frac{p-1}{k}\right) \left(m+\frac{p-1}{k}-1\right)\cdots \left(\frac{p-1}{k}+l+1\right)}{(m-l)!} \notag\\ \equiv & \frac{(km-1)(k(m-1)-1) \cdots (k(m-l+1)-1)}{k^l \cdot l!} \notag\\
    &\hspace{1cm}+ \frac{(km-1)(k(m-1)-1) \cdots (k(l+1)-1)}{k^{m-l}\cdot (m-l)!} \pmod{p}\notag \\      \equiv& \frac{(km-1) \cdots (k(m-l+1)-1)}{(km-kl)!_{(k)}}\notag\\&\hspace{2cm}\Bigg[\frac{(km-kl)!_{(k)}}{(kl)!_{(k)}} + \frac{(km-kl-1)!_{(k)}}{(kl-1)!_{(k)}}\Bigg] \pmod{p}.
\end{align}
Noting that $$p \nmid (km-kl)!_{(k)}$$ for $l= 0,1, \cdots, \frac{m-1}{2}$ and $$p \nmid (km-1) \cdots (k(m-l+1)-1)$$ for $l =1,2,\cdots, \frac{m-1}{2},$ we complete the proof the second part of the theorem because of \eqref{eqif1}, \eqref{eqif2}, and \eqref{eqif3}.
\par For any fixed $m$ and $k$, it is obvious that there are always finite number of primes $p\leq km+1$ such that $$p\mid S_{m+\frac{p-1}{k},k}(-1,p).$$ If $p>mk+1$, then we have already deduced that  
$$p \ |\ S_{m+\frac{p-1}{k},k}(-1,p)$$ if and only if $$p\ |\ \{(km)!_{(k)}+(km-1)!_{(k)}\}\displaystyle\prod_{l=1}^{\frac{m-1}{2}} \Bigg[\frac{(km-kl)!_{(k)}}{(kl)!_{(k)}} + \frac{(km-kl-1)!_{(k)}}{(kl-1)!_{(k)}}\Bigg].$$ Clearly, $$\{(km)!_{(k)}+(km-1)!_{(k)}\}\displaystyle\prod_{l=1}^{\frac{m-1}{2}} \Bigg[\frac{(km-kl)!_{(k)}}{(kl)!_{(k)}} + \frac{(km-kl-1)!_{(k)}}{(kl-1)!_{(k)}}\Bigg]$$ is a non-zero integer that is independent of $p$, hence it must have finite number of prime divisors. 
Hence the result follows.
 \end{pf}
\section{Conflicts of interest/Competing interests} The authors do not have any conflict of interest.
\section{Data availability}
Data sharing is not applicable to this article as no datasets were generated or analysed during the current study.
\section{Acknowledgement:} The first author acknowledges the support received from Council of Scientific \& Industrial Research (CSIR), India, through a Fellowship 09/1312(15921)/2022-EMR-I. The second author is supported by a project (CRG/2023/000482) of SERB, Department of Science and Technology, Goverment of India, under Core Research Grant.  

\end{document}